\documentclass[12pt]{amsart}
\usepackage{amssymb}
\usepackage{amsthm}
\usepackage{hyperref}
\usepackage{color}
\usepackage[mathscr]{eucal}
\usepackage{stmaryrd}
\usepackage{algorithm}
\usepackage[noend]{algpseudocode}
\usepackage{tikz} 
\usetikzlibrary{arrows,snakes,backgrounds} 

\newtheorem{thm}{Theorem}[section]
\newtheorem{lem}[thm]{Lemma}
\newtheorem{prop}[thm]{Proposition}
\newtheorem{cor}[thm]{Corollary}

\theoremstyle{remark}
\newtheorem*{cla}{\textit{Claim}}

\theoremstyle{definition}

\newtheorem{dfn}[thm]{Definition}

\newtheorem{probs}[thm]{Problems}

\newcommand{\bb}[1]{\mathbb{#1}}
\newcommand{\cc}[1]{{\mathscr #1}}

\newcommand{\alga}{{\bf A}}
\newcommand{\algb}{{\bf B}}

\newcommand{\algd}{{\bf D}}
\newcommand{\alge}{{\bf E}}

\newcommand{\algl}{{\bf L}}
\newcommand{\algp}{{\bf P}}
\newcommand{\algq}{{\bf Q}}
\newcommand{\algr}{{\bf R}}

\newcommand{\algt}{{\bf T}}

\newcommand{\meet}{\wedge}
\newcommand{\join}{\vee}

\newcommand{\la}{\langle}
\newcommand{\ra}{\rangle}
\newcommand{\dig}{\longrightarrow}
\newcommand{\geqstrong}{\geq_{1}}
\newcommand{\geqcong}{\geq_{2}}
\newcommand{\adj}{{standard $(2,3)$-instance}}

\newcommand{\leqsd}{\leq_{\sd}}
\DeclareMathOperator{\csp}{CSP}

\DeclareMathOperator{\alg}{Alg}

\DeclareMathOperator{\pr}{pr}

\DeclareMathOperator{\conn}{Con}
\DeclareMathOperator{\sd}{sd}

\title[2-semilattice over edge]{A constraint satisfaction problem algorithm for certain $2$-semilattice-over-edge algebras}

\author{Ian Payne}
\address{Department of Pure Mathematics, University of Waterloo, Waterloo, Ontario}
\email{ipayne@uwaterloo.ca}
\date{\today}
\begin{document}
\thanks{The author was supported by the Natural Sciences and Engineering Research Council of Canada.}

\begin{abstract}
To any fixed, finite relational structure, $\bb{D}$, there is an associated decision problem, $\csp(\bb{D})$, which is a restricted version of the constraint satisfaction problem. In \cite{bjk}, the so called ``algebraic approach" to the constraint satisfaction problem was established. The authors showed that to any finite relational structure, there is a corresponding finite algebra, and that the complexity of $\csp(\bb{D})$ depends only on this algebra. Therefore, they associate a decision problem, $\csp(\algd)$ to an algebra, $\algd$, and ignore the relational structure. Their ``algebraic dichotomy conjecture" suggests that a technical condition on $\algd$ implies $\csp(\algd)$ has a polynomial time algorithm. A significant sub-problem is the case when some reduct of $\algd$ has a congruence, $\theta$ so that $\algd/\theta$ has operations implying the local consistency algorithm correctly solves $\csp(\algd/\theta)$, and each $\theta$-equivalence class, $B$, has operations implying the few subpowers algorithm correctly solves $\csp(\algb)$. We give an algorithm for the case when $\algd$ has a binary term operation which is a $2$-semilattice operation on some quotient, $\algd/\theta$ of $\algd$, a projection on each $\theta$-class, and two other technical conditions are satisfied. Using this, we confirm the conjecture in the case that $\algd$ is in the join of two varieties, one of which has an edge term and the other is term equivalent to the variety of $2$-semilattices.
\end{abstract}

\maketitle

\section{Introduction}

Mathematicians have long been interested in computational problems involving assigning values to variables subject to constraints. For example, graph colouring was already of interest as early as the late 19th century, and boolean satisfiability has been of interest for well over half a century. People have been interested in scheduling problems since well before either of these. It was in 1974 in \cite{montanari} that Montanari wrote down a framework that could precisely describe many problems that involve finding solutions to constraints. This general form of the problem consists of a finite set, $X$, of variables, a finite set, $D$, of values, and a finite set of constraints restricting the values of tuples of variables from $X$. More precisely, a constraint is a pair $({\bf x},R)$ where ${\bf x}\in X^n$ and $R\subseteq D^n$. A function $\varphi:X\to D$ satisfies $({\bf x},R)$ if $(\varphi(x_1),\dots,\varphi(x_n))\in R$. We call the triple $(X,D,\cc{C})$ where $\cc{C}$ is a set of constraints an {\it instance}, and a solution to an instance is a function $\varphi:X\to D$ which simultaneously satisfies each constraint in $\cc{C}$.

It's not hard to encode problems such as graph $3$-colourability and boolean $3$-satisfiability in this framework. It follows that the decision problem which takes an instance and outputs YES if there is a solution, and NO otherwise is {\bf NP}-complete. Assuming ${\bf P}\neq{\bf NP}$, the problem can be made more interesting by restricting the types of constraints allowed in an instance. Roughly speaking, we consider a new decision problem where the question is the same, but a domain, $D$, and a set, $\cc{R}$ of allowed relations on $D$ are fixed. Instances of this decision problem are of the form $(X,D,\cc{C})$ where relations mentioned in the constraints in an instance must come from $\cc{R}$. For example, if we take $D=\{0,1\}$ and the only allowed relation is $R=\{(0,1),(1,0)\}$, then the associated decision problem is equivalent to graph $2$-colourability. This problem is well known to be solvable in polynomial time. However, if $D=\{0,1,2\}$ and the only allowed relation is $S=\{(0,1),(1,0),(0,2),(2,0),(1,2),(2,1)\}$, then the associated decision problem precisely encodes graph $3$-colourability, which is known to be ${\bf NP}$-complete. Therefore, if ${\bf P}\neq{\bf NP}$, there is some content to the question which asks about the complexity of such a restricted decision problem.

The major unsettled conjecture in the area was made by Feder and Vardi in \cite{feder-vardi}. Their conjecture is that each such restricted decision problem is in ${\bf P}$ or it is ${\bf NP}$ complete. It is known as the dichotomy conjecture, and it is nontrivial if ${\bf P} \neq {\bf NP}$. The conjecture was made based on evidence from dichotomy theorems of Hell and Ne\v{s}et\v{r}il in \cite{hell-nesetril} and Schaefer in \cite{schaefer}. There have since been dichotomy theorems in other special cases. Some of these results are summarized later in this section.

In the late 1990s, Jeavons and a variety of coauthors in several papers including \cite{bj}, \cite{bjk}, \cite{jc}, and \cite{jcp} established a connection between constraint satisfaction problems and universal algebra. They also stated a stronger dichotomy conjecture known as the algebraic dichotomy conjecture. It implies the conjecture of Feder and Vardi, and has provided a very popular approach to dichotomy conjecture. The algebraic approach to the constraint satisfaction problem takes advantage of the structure of the ``algebra of polymorphisms" of the underlying relational structure. The discoverers defined the class of ``tractable" algebras, a class of finite algebras closed under taking subalgebras, homomorphic images, finite powers, and term expansions. A class of algebras (in particular, a variety) is called tractable if each of its finite members is tractable. They also associated to each finite relational structure, $\bb{D}$, a finite algebra, $\alg(\bb{D})$ with the property that $\csp(\bb{D})$ is solvable in polynomial time if and only if $\alg(\bb{D})$ is tractable. In \cite{bjk}, it is shown that if every operation of $\alg(\bb{D})$ satisfies $f(x,x,\dots,x)\approx x$ (is idempotent), and $\alg(\bb{D})$ has no ``Taylor operation", then $\csp(\bb{D})$ is ${\bf NP}$-complete. The algebraic dichotomy conjecture states that if every operation of $\alg(\bb{D})$ is idempotent and it has a Taylor operation (defined in Section~\ref{algebrasection}), then $\csp(\bb{D})$ is solvable in polynomial time. Equivalently, if $\alg(\bb{D})$ has a Taylor operation, then it is tractable.

Since the problem was posed in \cite{bjk}, there have been many partial confirmations. The two most general are when $\alg(\bb{D})$ is idempotent and has an ``edge" operation, and when $\alg(\bb{D})$ has operations which imply ``congruence meet semi-dristributivity". See \cite{hobbymckenzie} for a formal definition of the latter. In the first case, $\alg(\bb{D})$ is shown to be tractable by the ``few subpowers" algorithm, which is a generalization of Gaussian elimination. This most general case was proved gradually by various authors in \cite{bulatov-dalmau},\cite{dalmau}, \cite{fewsubpowers1}, and \cite{fewsubpowers2}. In the latter, $\alg(\bb{D})$ is shown to be tractable by ``local consistency checking". This algorithm checks an input to $\csp(\bb{D})$ for local solutions, which, in this case, implies the existence of a global solution. In \cite{larose-zadori}, Larose and Z\'adori showed that if $\alg(\bb{D})$ is tractable via the local consistency checking algorithm, then $\alg(\bb{D})$ is in a congruence meet semi-distributive variety. Barto and Kozik proved the converse in \cite{bkbounded}.

This paper studies a class of algebras which have a Taylor operation but do not fall in either of the cases described above. In particular, we consider finite algebras, $\algd$, which have a congruence, $\theta$, so that $\algd/\theta$ has a $2$-semilattice operation (defined in the next section), and the $\theta$-classes belong to a common tractable variety. It should be noted that having a $2$-semilattice operation implies congruence meet semi-distributivity. More precisely, the main result is the following.

\begin{thm}\label{mainthm}
Let $\cc{W}$ be a tractable variety, and $\algd$ be a finite algebra similar to $\cc{W}$. Suppose $\algd$ has a binary term, $\cdot$, and a congruence, $\theta$ such that $\cdot$ is a $2$-semilattice operation on $\algd/\theta$, and each $\theta$-class as a subalgebra of $\algd$ is in $\cc{W}$. Also suppose the following hold:
\begin{enumerate}
\item $\cc{W}\vDash x\cdot y\approx x$
\item $\algd\vDash x\cdot(y\cdot z)\approx x\cdot(z\cdot y)$.
\end{enumerate}
Then $\algd$ is tractable.
\end{thm}

Theorem~\ref{mainthm} can be used to prove the following more concrete corollary.

\begin{cor}\label{maincor}
Suppose $\cc{W}$ and $\cc{T}$ are similar idempotent varieties. If $\cc{W}$ has an edge term and $\cc{T}$ is term equivalent to the variety of $2$-semilattices, then $\cc{W}\join\cc{T}$ is tractable.
\end{cor}

The proof of Theorem~\ref{mainthm} closely follows Bulatov's proof from \cite{bulatov1} that finite $2$-semilattices are tractable.

Section~\ref{algebrasection} is meant to be an introduction to universal algebra. If the reader finds it necessary, a much more comprehensive introduction to universal algebra can be found in either \cite{bergman} or \cite{burrissanka}. Both are well written introductions to the subject, but the notation in this paper will more closely follow the former than the latter. Section~\ref{cspsection} is a formal introduction to the constraint satisfaction problem ($\csp$). Hopefully it will indicate how the algebraic approach to the $\csp$ is a very natural way to study the purely combinatorial problem vaguely introduced in the first paragraph of this introduction. Section~\ref{bulatovsection} goes through a proof of Bulatov's result from \cite{bulatov1}, and Section~\ref{bulatovsolutionssection} describes what we have called ``Bulatov solutions" to certain $\csp$ instances. Section~\ref{mainresult} includes a proof of Theorem~\ref{mainthm}, and Section~\ref{applicationsection} contains a proof of Corollary~\ref{maincor}.

\section{Algebra}\label{algebrasection}
For a set, $A$, an $n$-ary operation on $A$ is a function, $f:A^n\to A$. The integer $n$ is called the arity of $f$, and we say that $f$ is $n$-ary. When $n=1,2$, or $3$, we use the words unary, binary, and ternary, respectively. An algebra is a pair, $\alga=(A;F)$ where $A$ is a nonempty set called the universe of ${\bf A}$, and $F$ is a set of operations on $A$. We will always denote the universe by a capital letter and the algebra by the corresponding bold letter. The algebras in this paper will usually be ``indexed", which means the set of operations is indexed in the following sense. We fix a symbol set, $\cc{F}$, and a function $\rho:\cc{F}\to\bb{N}$ which assigns arty. An algebra is said to have similarity type $\rho$ if its set of operations is $\cc{F}^{\alga}=( f^{\alga}:f\in\cc{F})$ where $f^{\alga}$ is some $n$-ary operation on $A$ where $\rho(f)=n$. We will often omit the word ``similarity" and simply say ``type". As well, we will often think of a similarity type as a set of symbols with built-in arities and say things like ``$f$ is in the type of $\alga$" rather than ``$f$ is in the domain of the type of $\alga$". The superscripts will usually be omitted, unless the algebra is not clear from context. An algebra, $\alga$, is said to be finite if $A$ is finite, and $\alga$ is said to have finite type if $\cc{F}$ is finite. In the case that $\cc{F}$ is finite, we will often denote an algebra of type $\rho$ simply by listing the operations after the semicolon. For example, if $\cc{F}=\{f,g,h\}$, we may write $\alga=(A;f^{\alga},g^{\alga},h^{\alga})$ or $\alga=(A;f,g,h)$ rather than $\alga=(A,(f^{\alga},g^{\alga},h^{\alga}))$.

We say that two algebras are similar if they have the same similarity type. For example, all groups are similar since their similarity type can be thought to have symbol set $\cc{F}=\{\cdot,^{-1},e\}$ with $\rho(\cdot)=2$, $\rho(^{-1})=1$, and $\rho(e)=0$. Each individual group then has its own binary, unary, and zero-ary operations. For an algebra, $\alga$, of type $\rho:\cc{F}\to\bb{N}$, we call $\cc{F}^{\alga}$ its set of basic operations. Its term operations informally consist of the basic operations, all projections, and all operations obtained by composition of these. More precisely, a term operation is the natural interpretation as an operation of some term in the symbol set from the similarity type. For example, if $\rho$ has a binary symbol, $f$, and a ternary symbol, $g$, then the symbol set from $\rho$ has a term $fx_4gx_1x_3fx_5gx_2x_2x_3$. This gives rise to a $5$-ary operation called a term operation, $h^{\alga}$ defined by $h^{\alga}(x_1,x_2,x_3,x_4,x_5)=f^{\alga}(x_4,g^{\alga}(x_1,x_3,f^{\alga}(x_5,g^{\alga}(x_2,x_2,x_3))))$ for every algebra, $\alga$, of type $\rho$.

For an algebra of type $\rho:\cc{F}\to\bb{N}$, a subuniverse of $\alga$ is a subset, $B\subseteq A$ which is closed under all operations in $\cc{F}^{\alga}$. A subalgebra of $\alga=(A;\cc{F}^{\alga})$ is an algebra $\algb=(B;\cc{F}^{\algb})$, similar to $\alga$, where $B$ is a nonempty subuniverse of $\alga$, and $f^{\algb}=f^{\alga}\!\!\upharpoonright_{B}$ for each symbol, $f\in\cc{F}$. We write $B\leq \alga$ if $B$ is a subuniverse of $\alga$, and $\algb\leq\alga$ if $\algb$ is a subalgebra of $\alga$. If $\alga$ and $\algb$ are similar algebras, a homomorphism from $\alga$ to $\algb$ is a function $\varphi:A\to B$ such that for every $n$-ary operation symbol, $f$ of arity $n$, and any $a_1,\dots,a_n\in A$, we have $\varphi(f^{\alga}(a_1,\dots,a_n))=f^{\algb}(\varphi(a_1),\dots,\varphi(a_n))$. In the case that $\varphi$ is surjective, we call $\algb$ a homomorphic image of $\alga$, and if $\varphi$ is a bijection, we call it an isomorphism and say that $\alga$ and $\algb$ are isomorphic. For a family, $(\alga_u:u\in U)$ of similar algebras, the product $\alga=\prod_{u\in U}\alga_u$ has universe $\prod_{u\in U}A_u$, which is formally the set of functions, $\sigma:U \to \bigcup_{u\in U} A_u$ with $\sigma(u)\in A_u$ for each $u$. For each $n$-ary operation symbol, $f$ in the type of the $\alga_u$, the operation $f^{\alga}$ is defined by $f^{\alga}(\sigma_1,\dots,\sigma_n)(u)=f^{\alga_u}(\sigma_1(u),\dots,\sigma_n(u))$. All products explicitly mentioned in this paper will be finite, but arbitrary products are important to the theory of universal algebra. When taking the product of finitely many algebras, the universe can be visualized as the usual cartesian product consisting of $n$-tuples, and the operations are coordinate-wise. We say that $\algr$ is a subdirect product of $(\alga_u:u\in U)$ if $\algr\leq\prod_{u\in U}\alga_u$ and $\pr_u(R)=A_u$ for every $u\in U$. In this case, we write $\algr\leqsd\prod_{u\in U}\alga_u$.

A congruence on an algebra, $\alga$, is an equivalence relation on its domain so that the operations of $\alga$ act in a well defined way on its equivalence classes. This means that an equivalence relation, $\theta$, on $A$ is a congruence on $\alga$ if for any $n$-ary operation, $f$ of $\alga$ and any $(a_1,b_1),\dots,(a_n,b_n)\in\theta$, we have that $(f(a_1,\dots,a_n),f(b_1,\dots,b_n))\in\theta$ as well. We will often use the notation $a\overset{\theta}{\equiv} b$ to denote $(a,b)\in\theta$. Congruences will usually be denoted by Greek letters, with a notable exception in the next sentence. Every algebra, $\alga$, has two congruences, $0_A=\{(a,a):a\in A\}$ and $1_A=A\times A$. In the case that $|A|=1$, these congruences coincide. Because of how congruences are defined, we can take a quotient of $\alga$ by a congruence, $\theta$, and obtain an algebra, $\alga/\theta$, similar to $\alga$. The universe is $A/\theta$ and for an $n$-ary operation symbol, $f$, the operation $f^{\alga/\theta}$ is defined by $f^{\alga/\theta}(a_1/\theta,\dots,a_n/\theta)=f^{\alga}(a_1,\dots,a_n)/\theta$. The operations being well-defined is precisely the condition that distinguishes congruences from other equivalence relations. It's not hard to see that for a homomorphism $\varphi:\alga\to\algb$, the kernel, defined by $\ker(\varphi)=\{(a_1,a_2)\in A^2:\varphi(a_1)=\varphi(a_2)\}$, is a congruence on $\alga$. Furthermore, if $\varphi$ is surjective, $\alga/\ker(\theta)\cong\algb$. As well, for any congruence $\theta$ on $\alga$, the kernel of the canonical homomorphism $\varphi:\alga\to\alga/\theta$ is $\theta$. Therefore, we get the usual correspondence between quotients and homomorphic images. For congruences, $\alpha$ and $\beta$ on an algebra, we define $\alpha\meet\beta=\alpha\cap\beta$ and $\alpha\join\beta$ to be the transitive closure of $\alpha\cup\beta$. Both are congruences, and the set of congruences forms a complete bounded lattice, known as the congruence lattice, with respect to these operations.

A variety in universal algebra is a class of similar algebras which is closed under taking products, subalgebras, and homomorphic images (quotients). By a result of Birkhoff \cite{birkhoff}, a variety is defined by some set of equations. This means that every variety is the class of algebras satisfying some fixed set of identities in its similarity type, and vice versa. An identity is a universally quantified equation. For example, consider an algebra, $\alga$ whose similarity type yields terms $s$ and $t$. We say that $\alga$ satisfies $s\approx t$ and write $\alga\vDash s\approx t$ if $s^{\alga}=t^{\alga}$. This may seem nonsensical as $s^{\alga}$ and $t^{\alga}$ need not even be functions of the same variables. For example, perhaps $s=fx_3fx_1x_2$ and $t=fx_3fx_5x_4$. In this case, $s^{\alga}\approx t^{\alga}$ means for any assignment $(a_1,a_2,a_3,a_4,a_5)\in A^5$ of $(x_1,x_2,x_3,x_4,x_5)$, we have that $s^{\alga}(a_1,a_2,a_3)=t^{\alga}(a_3,a_4,a_5)$. We say that a variety, $\cc{V}$, satisfies $s\approx t$ and write $\cc{V}\vDash s\approx t$ if every algebra in $\cc{V}$ satisfies it. Groups are a familiar example of a variety axiomatized by a short list of identities. The class of all groups is a variety since it is closed under taking products, subgroups, and homomorphic images. It is also precisely the class of algebras with similarity type given by the symbol set $\{\cdot,^{-1},e\}$ having arities $2,1, $ and $0$, respectively, that satisfy the group identities.

Also in \cite{birkhoff}, Birkhoff proved that the variety generated by a set, $\cc{K}$ of similar algebras (defined to be the intersection of all varieties containing $\cc{K}$) is exactly the class of homomorphic images of subalgebras of products of members of $\cc{K}$. The variety generated by $\cc{K}$ will be denoted by $V(\cc{K})$. For similar varieties, $\cc{V}$ and $\cc{W}$, we denote by $\cc{V}\join\cc{W}$ the variety generated by their union.

A $2$-semilattice operation is a binary operation, $\cdot$, satisfying
\begin{enumerate}
\item $x\cdot x \approx x$,
\item $x\cdot y \approx y\cdot x$, and
\item $x\cdot(x\cdot y) \approx x\cdot y$.
\end{enumerate}
We call an algebra a $2$-semilattice if its only basic operation is a $2$-semilattice operation. The origin of $2$-semilattices seems somewhat mysterious. They were mentioned by Quackenbush in 1995 in \cite{quackenbush}, where he alluded to \cite{jezekquack} which is an earlier paper about the variety of directoids. The variety of directoids is strictly between the variety of semilattices and the variety of $2$-semilattices, but \cite{jezekquack} does not include the word ``$2$-semilattice" anywhere. Both \cite{jezekquack} and \cite{quackenbush} are about minimal clones. It seems that some of the earliest interest in $2$-semilattices was because of their relevance in clone theory.

The Maltsev product of two varieties is an idea introduced by Maltsev in \cite{maltsev}. The following is less general than Maltsev's original definition.

\begin{dfn}
Let $\cc{A}$ and $\cc{B}$ be idempotent varieties of the same type. The Maltsev product of the two varieties, denoted $\cc{A}\circ\cc{B}$, is the class of all idempotent algebras which have a congruence whose classes (as subalgebras) are all in $\cc{A}$, and the quotient by it is in $\cc{B}$. If $\alga$ is an algebra and $\theta\in\conn(\alga)$ has these properties, we will say that it witnesses $\alga\in\cc{A}\circ\cc{B}$.
\end{dfn}

Maltsev's original definition did not require the classes to be varieties and made no mention of idempotence. Idempotence guarantees that congruence classes are subuniverses, so it makes the definition easier to work with.

To finish off this section, we will state, without proof, a deep result of Barto and Kozik which first appeared as Theorem~2.3 in \cite{bkcyclic}. Before stating Theorem~\ref{absorptiontheorem}, commonly known as the Absorption Theorem, we need to introduce some terminology.

Taylor operations, which came up in the statement of the algebraic dichotomy conjecture, were first described by Taylor in \cite{taylor}. We now define them precisely.

\begin{dfn}\label{taylordfn}
Let $A$ be a set and $n\geq 2$. A Taylor operation of arity $n$ is an idempotent operation $t:A^n\to A$ which, for each $i\leq n$, satisfies an identity $t(x_1,x_2,\dots,x_n)\approx t(y_1,y_2,\dots,y_n)$ where $x_k,y_k\in\{x,y\}$ and $x_i\neq y_i$. We say that $\alga$ has a Taylor term operation if there is some term, $t$, in its similarity type such that $t^{\alga}$ is a Taylor operation on $A$. We say that a term, $t$, is a Taylor term for a variety, $\cc{V}$, if it is idempotent in $\cc{V}$ and there is a set of identities of the kind described above that hold in $\cc{V}$.
\end{dfn}

The next definition is that of an absorbing subuniverse. Examples of absorbing subuniverses are ideals of rings without 1 and every singleton subuniverse of a lattice.

\begin{dfn}\label{absorptiondfn}
Let $\alga$ be an algebra and $B\leq \alga$ be a nonempty subuniverse of $\alga$. We say that $B$ is absorbing if there is an $n$-ary term operation, $f$ of $\alga$ with $n\geq 2$ such that $f(b_1,\dots,b_n)\in B$ whenever at least $n-1$ of the $b_i$ are in $B$. When $B$ is an absorbing subuniverse of $\alga$, we write $B\vartriangleleft \alga$.
\end{dfn}

Of course, the set $A$ is an absorbing subuniverse of $\alga$, so we call $B$ a proper absorbing subuniverse if it is not all of $A$. As mentioned, every singleton subset of a lattice is absorbing. If $\algl$ is a bounded lattice with bottom element $0$ and $|L|\geq 2$, then $\{0\}$ is a proper absorbing subuniverse with respect to the term $\meet$. Similarly, the top element is absorbing with respect to $\join$. In fact, every singleton subuniverse is absorbing with respect to the term operation $(x\meet y)\join(y\meet z)\join(z\meet x)$. We say that an algebra $\alga$ is absorption free if it has no proper absorbing subuniverse.

Let $\alga$ and $\algb$ be finite algebras with $R\leqsd\alga\times\algb$. We can visualize $R$ as a bipartite graph whose vertex set is $A\cup B$, and which has an edge between $a$ and $b$ exactly when $(a,b)\in R$. We say that $R$ is linked when this graph is connected.

Now we state the Absorption Theorem of Barto and Kozik.

\begin{thm}[Theorem 2.3 in \cite{bkcyclic}]\label{absorptiontheorem}
Suppose $\alga$ and $\algb$ are finite, absorption free, and in an idempotent variety with a Taylor term. If $R\leqsd\alga\times\algb$ is linked, then $R=A\times B$.
\end{thm}

\section{Constraint satisfaction problems}\label{cspsection}
In this section, we will give precise definitions regarding constraint satisfaction problems.

\subsection{Three versions of $\csp$}\label{cspsec1}
\begin{dfn}\label{cspdef}\hspace{1cm}
\begin{enumerate}
\item An instance is a triple $\cc{I}=(X,D,\cc{C})$ where $X$ is a finite set of variables, $D$ is a finite set, and $\cc{C}$ is a finite set of constraints. A constraint, $C\in\cc{C}$ is a pair $({\bf x},R)$ where ${\bf x}\in X^n$ and $R\subseteq D^n$ for some $n$.
\item A solution to an instance is a function $\varphi:X\to D$ with  $(\varphi(x_1),\dots,\varphi(x_n))\in R$ for every $({\bf x},R)\in\cc{C}$.
\item The constraint satisfaction problem, abbreviated $\csp$, is the decision problem whose input is an instance, $\cc{I}$, and output is YES if a solution exists, and NO otherwise.
\end{enumerate}
\end{dfn}

It's not hard to see that $\csp$ is in ${\bf NP}$, and it was mentioned in the introduction that graph $3$-colourability can be encoded in this way. Since $3$-colourability is known to be ${\bf NP}$-complete, it follows that $\csp$ is ${\bf NP}$-complete. The question of complexity is more interesting for versions of $\csp$ where the allowed instances are restricted.

\begin{dfn}
A relational structure is a pair $\bb{D}=(D,\cc{R})$ where $D$ is a set and $\cc{R}$ is a set of relations on $D$. We say $\bb{D}$ is finite if $D$ is finite, and has finite type if $\cc{R}$ is finite.
\end{dfn}

With this in mind, we can define a less general version of $\csp$.

\begin{dfn}
Fix a finite relational structure, $\bb{D}=(D,\cc{R})$, of finite type. The decision problem $\csp(\bb{D})$ is the same as the general $\csp$, except the input is restricted to instances, $\cc{I}=(X,D,\cc{C})$ where $D$ is the domain of $\bb{D}$, and for each constraint, $({\bf x},R)\in\cc{C}$, the relation $R$ is in $\cc{R}$.
\end{dfn}

While the general $\csp$ is easily seen to be ${\bf NP}$ complete, more interesting results are known about this restricted $\csp$. For example, Schaeffer proved in \cite{schaefer} that if $\bb{D}=(D,\cc{R})$ has $|D|=2$, then $\csp(\bb{D})$ is either in ${\bf P}$ or is ${\bf NP}$-complete. A similar result was obtained by Hell and Ne\v{s}et\v{r}il in \cite{hell-nesetril} when $\cc{R}$ consists of a single binary relation which is the edge relation of a simple graph (symmetric and irreflexive). This version of $\csp$ was the subject of the dichotomy conjecture of Feder and Vardi: For a finite relational structure, $\bb{D}$, of finite type, either $\csp(\bb{D})$ is in ${\bf P}$ or it is ${\bf NP}$-complete.

We have two versions of $\csp$ so far: the general $\csp$ and $\csp(\bb{D})$ for a fixed relational structure, $\bb{D}$. The next one is a special case of $\csp(\bb{D})$.

\begin{dfn}\label{algcsp}
Fix a finite idempotent algebra, $\algd$, and a positive integer, $n$. Define $\cc{R}_n(\algd)=\{A:A\leq \algd^m\text{ for some }m\leq n\}$. The decision problem $\csp(\algd,n)$ is the decision problem $\csp(\bb{D})$ where $\bb{D}=(D,\cc{R}_n)$.
\end{dfn}

From Definition~\ref{algcsp}, we can now precisely define what it means for an algebra to be tractable.

\begin{dfn}\label{tractabledfn}
A finite idempotent algebra is tractable if for every $n\geq 2$, there is a polynomial time algorithm which solves $\csp(\algd,n)$.
\end{dfn}

The algebraic dichotomy conjecture from \cite{bjk} states that if an idempotent algebra, $\algd$, has a Taylor operation, then it is tractable. We give a brief explanation as to why this conjecture implies the dichotomy conjecture of Feder and Vardi from \cite{feder-vardi}. For a more detailed explanation of this, see \cite{bjk}.

For any finite relational structure, $\bb{D}$, there is known to be another relational structure, $\bb{D}'$ which is a ``core", so that $\csp(\bb{D})$ and $\csp(\bb{D}')$ reduce to one another in polynomial time. The structure, $\bb{D}'$ is essentially unique, and if $\bb{D}$ is already a core, then $\bb{D}=\bb{D}'$. For a relational structure, $\bb{D}=(D,\cc{R})$, we denote by $\bb{D}^{\text{c}}$ the structure $(D,\cc{R}\cup\{\{c\}:c\in D\})$. If $\bb{D}$ is a core, then $\csp(\bb{D})$ and $\csp(\bb{D}^{\text{c}})$ are polynomial time reducible to one another. It is at this point that we can pass to algebras. For a relation, $R$ on a set $D$, we say that a function $f:D^k\to D$ preserves $R$ if for any ${\bf x}_1,\dots,{\bf x}_k\in R$ we have $f({\bf x}_1,\dots,{\bf x}_k)\in R$ as well. In the previous sentence, $f$ is applied coordinate-wise. We call such $f$ a polymorphism of $\bb{D}$ if it preserves all of its relations. Notice that if $\bb{D}=\bb{D}^{\text{c}}$, then every polymorphism of $\bb{D}$ is idempotent. If we let $\algd$ be the algebra with domain $D$ and whose operations are the polymorphisms of $\bb{D}$, it can be shown that $\csp(\bb{D})$ and $\csp(\algd,n)$ are polynomial time equivalent where $n$ is the maximum arity of the relations in $\bb{D}$. To summarize, if we start with any finite relational structure of finite type, $\bb{D}$, there is a finite idempotent algebra, $\algd$, and an integer $n\geq 2$ such that $\csp(\bb{D})$ is polynomial time equivalent to $\csp(\algd,n)$. In the case that $\algd$ has no Taylor operation, Bulatov, Jeavons, and Krokhin showed in \cite{bjk} that $\csp(\algd,n)$ is ${\bf NP}$-complete for all $n\geq 2$, so $\csp(\bb{D})$ must be ${\bf NP}$-complete. Their conjecture says that if $\algd$ has a Taylor operation, then $\csp(\algd,n)$ has a polynomial time algorithm for all $n\geq 2$. This would imply that $\csp(\bb{D})$ has a polynomial time algorithm, as well.

\subsection{$(2,3)$-consistency}
We now introduce a subproblem of $\csp(\algd,2)$. After this section, it will be the only version of $\csp$ dealt with in this paper. No generality is lost by considering only this sub-problem. The following proposition is stated without proof, but a proof can be found in \cite{bkbounded}.

\begin{prop}\label{reducetotwo}
Let $\algd$ be a finite idempotent algebra and $n\geq 2$ be an integer. Then $\csp(\algd,n)$ is polynomial time equivalent to $\csp(\algd^{\lceil\frac{n}{2}\rceil},2)$.
\end{prop}

Because of Proposition~\ref{reducetotwo}, we will only consider instances of $\csp(\algd,2)$. Before proceeding, we introduce the notation $R^{-1}$ to denote the binary relation $\{(b,a):(a,b)\in R\}$ for a binary relation, $R$.

\begin{dfn}\label{ourcsp}
Let $\algd$ be a finite idempotent algebra and define $\csp(\algd)$, a subproblem of $\csp(\algd,2)$ as follows: We restrict inputs, $(X,D,\cc{C})$, to have the property that $\cc{C}=\{(x,P_x):x\in X\}\cup\{((x,y):R_{x,y}):(x,y)\in X\times X\}$ for some subuniverses, $P_x$ and $R_{x,y}$ of $D$ and $D\times D$, respectively, so that $R_{x,y}\subseteq P_x\times P_y$ for every $(x,y)\in X\times X$. We will refer to an instance as $\cc{I}=(X,\cc{P},\cc{R})$ where, $\cc{P}=(P_x:x\in X)$ and $\cc{R}=(R_{x,y}:(x,y)\in X^2)$. That is, we identify it by its set $\cc{P}$ of {\it potatoes} and $\cc{R}$ of {\it relations}. We call $\cc{I}$ a \adj{} if it also satisfies the following four conditions.
\begin{enumerate}
\item[(P1)] For each $x\in X$, $R_{x,x}=0_{P_x}$,
\item[(P2)] For $x,y,z\in X$ and any $(a,b)\in R_{x,y}$, there is a $c\in P_z$ such that $(a,c)\in R_{x,z}$ and $(b,c)\in R_{y,z}$,
\item[(P3)] For each $x,y\in X$, $R_{x,y}\leqsd\algp_x\times\algp_y$ if $P_x$ and $P_y$ are both non-empty.
\item[(P4)] $R_{y,x}=R_{x,y}^{-1}$ for each $x,y\in X$.
\end{enumerate}
\end{dfn}

We note that the set of solutions to a \adj{}, if nonempty, can be identified as a subuniverse of the product $\prod_{x\in X}\algp_x$. For this reason, when $\cc{I}$ has at least one solution, we will sometimes refer to its algebra of solutions and use the fact that it is in the variety generated by $\algd$. It is not hard to see that for a \adj{}, there is an empty potato if and only if all potatoes (and hence, relations) are empty. We call such an instance empty. It is worth noting that (P3) and (P4) follow from (P1) and (P2). Therefore, to prove an instance is a \adj{}, we need only verify (P1) and (P2). We leave (P3) and (P4) in the definition because they are part of the intuition behind \adj s and we would need to derive them for later use anyway.

The $(2,3)$-consistency checking algorithm (Algorithm~\ref{23minalg}) takes any instance of $\csp(\algd,2)$ as input and outputs a \adj{} of $\csp(\algd)$. The output \adj{} has the same solutions as the input instance. It's not hard to see that $(2,3)$-consistency checking runs in polynomial time. For more explanation on this, see \cite{bulatov1}, \cite{feder-vardi}, or \cite{larose-zadori}. Because of this, if we wish to find a polynomial time algorithm for $\csp(\algd,2)$, it is enough to find one for $\csp(\algd)$ with inputs restricted to \adj s.

From now on we will write $\algr_{xy}$ and $R_{xy}$ (omitting the comma) when talking about relations. When two instances, $\cc{I}$ and $\cc{J}$ have the same set of variables, $X$, we use a superscript to indicate the instance to which a potato or relation belongs. When $\algp_{x}^{\cc{J}}\leq\algp_{x}^{\cc{I}}$ and $\algr_{xy}^{\cc{J}}\leq\algr_{xy}^{\cc{I}}$ for every $x,y\in X$, we say that $\cc{J}$ is a subinstance of $\cc{I}$.

\begin{algorithm}
\caption{$(2,3)$-consistency checking}\label{23minalg}
\begin{algorithmic}[1]
\State {\bf Input}: An instance, $\cc{I}=(X,D,\cc{C})$ of $\csp(\algd,2)$
\For {$(x,y)\in X^2$}
\If {$x\neq y$}
\State $R_{xy} \gets D\times D$
\Else
\State $R_{xy} \gets 0_D$
\EndIf
\EndFor

\For {$((x,y),R)\in\cc{C}$}
\State $R_{xy} \gets R_{xy}\cap R$
\State $R_{yx} \gets R_{yx}\cap R^{-1}$
\EndFor

\State $\text{flag} \gets 1$
\While {$\text{flag} = 1$}
\State $\text{flag} \gets 0$
\For {$(x,y)\in X^2$}
\For {$z\in X$}
\For {$(a,b)\in R_{xy}$}
\If {there is no $c\in D$ with $(a,c)\in R_{xz}$ and $(b,c)\in R_{yz}$}
\State $R_{xy} \gets R_{xy}\setminus\{(a,b)\}$
\State $R_{yx} \gets R_{yx}\setminus\{(b,a)\}$
\State $\text{flag} \gets 1$
\EndIf
\EndFor
\EndFor
\EndFor
\EndWhile

\For {$x\in X$}
\State $P_x=\pr_1(R_{xx}))$
\EndFor

\State {\bf Output}: $(X,(P_x:x\in X),(R_{xy}:(x,y)\in X^2))$
\end{algorithmic}
\end{algorithm}

\begin{prop}
The output of Algorithm~\ref{23minalg} is a \adj{} of $\csp(\algd)$ with exactly the same set of solutions as the input instance.
\end{prop}

\section{$2$-semilattices and Maltsev Products}\label{varsection}
We begin this section with a basic result on $2$-semilattices. It is really a list of observations about a useful digraph structure that a finite $2$-semilattice admits. Most of Lemma~\ref{digraphproperties} appeared either implicitly or explicitly at some point in \cite{bulatov1}. The variety of $2$-semilattices, defined in Section~\ref{algebrasection}, has only a binary operation symbol, $\cdot$, in its similarity type. We denote this variety by $\cc{S}$.

\begin{dfn}\label{digraphdfn}
For each $\alga\in\cc{S}$, define a digraph relation on $A$ by $a\overset{\alga}{\dig} b$ if $a\cdot b=b$. Since the digraph relation depends on the operation, we indicate the algebra above the arrow. However, we will omit this whenever the algebra is clear from context.
\end{dfn}

\begin{lem}\label{digraphproperties}
Let $\alga\in\cc{S}$ be finite. The following hold for the digraph $(A,\overset{\alga}{\dig})$.
\begin{enumerate}
\item For any $a,b\in A$, $a\dig a$, $a\dig a\cdot b$ and $b\dig a\cdot b$.
\item If the strongly connected components are quasi-ordered by $U\geq V$ iff $u\dig v$ for some $u\in U$ and $v\in V$, there is a unique minimal component, $A'$. This component has the property that for any $b\in A$ there is $a\in A'$ such that $b\dig a$.
\item With $A'$ as above, $a\in A'$ if and only if for every $b\in A$ there is a directed walk from $b$ to $a$.
\item $A'$ is an absorbing subuniverse of $\alga$ with respect to $\cdot$.
\item If $a,b\in A$ with $a\dig b$, then $\la\{a,b\};\cdot\ra$ is a semilattice with absorbing element $b$.
\item Let $\alpha$ be a congruence on $\alga$. If $(A,\overset{\alga}{\dig})$ is strongly connected, then $(A/\alpha,\overset{\alga/\alpha}{\dig})$ is strongly connected.
\end{enumerate}
\end{lem}

We include a proof for (2) and (3)

\begin{proof}
\hspace{1cm}
\begin{enumerate}
\item[(2)] The relation $\psi=\{(a,b):a\text{ and }b\text{ are on a directed cycle}\}$ is an equivalence relation on $A$ whose equivalence classes are exactly the strongly connected components. Because of the way $\geq$ is defined, if $U\geq V_1\geq V_2\geq\cdots\geq V_n\geq U$, then $U=V_1=\cdots=V_n$. Therefore, by finiteness, there are minimal components in $(A,\dig)$. By minimality, any such component, $U$, has the property that if $u\in U$ and $u\dig v$, then $v\in U$. Suppose $U$ and $V$ are minimal components. Fix $u\in U$ and $v\in v$. From (1), $u\dig u\cdot v$ and $v\dig u\cdot v$. From the previous remark, we have $u\cdot v\in U\cap V$. Since $U$ and $V$ are classes of an equivalence relation, we get $U=V$. Now pick any $b\in A$ and $a'\in A'$ and set $a=b\cdot a'$. By the previous remark, $a\in A'$, and $b\dig a$ by (1).

\item[(3)] If $a\in A'$ and $b\in A$, there is $c\in A'$ such that $b\dig c$ by (2). Since $A'$ is strongly connected, there is a directed walk from $c$ to $a$, so there is a directed walk from $b$ to $a$. Conversely, suppose there is a directed walk from every vertex to $a$. In particular, for any $b\in A'$ there is a directed walk from $b$ to $a$. From the proof of (2), we have that any out-neighbour of a member of $A'$ is itself a member of $A'$, so $a$ is in $A'$.
\end{enumerate}
\end{proof}

The strongly connected component guaranteed by (2) will be referred to as the ``smallest strongly connected component" of $\alga$, and will be denoted by adding the superscript, $'$, to the universe as in the Lemma. We now continue with more facts about $2$-semilattices.

\begin{prop}
Every binary term, $t$ in the type of $\cc{S}$ which depends on both of its variables satisfies $\cc{S}\vDash t(x,y)\approx x\cdot y$.
\end{prop}

\begin{proof}
The proof can be carried out by induction on term ``height", in the sense of \cite{bergman}.
\end{proof}

\begin{cor}\label{termequivcor}
Suppose $\cc{T}$ is an idempotent variety which is term equivalent to $\cc{S}$ and $*$ is the binary term of $\cc{T}$ which is the image of $\cdot$ under this term equivalence. If $t$ is a binary term in the similarity type of $\cc{T}$ which depends on both variables, then $\cc{T}\vDash t(x,y)\approx x*y$.
\end{cor}

The following is an observation of Ross Willard.

\begin{prop}\label{maltsevprodvariety}
Let $\cc{A}$ and $\cc{B}$ be idempotent varieties of the same type and suppose the following hold:
\begin{enumerate}
\item There is a binary term, $t$ in the type of $\cc{A}$ and $\cc{B}$ so that $\cc{A}\vDash t(x,y)\approx x$ and $\cc{B}\vDash t(x,y)\approx t(y,x)$.
\item $\cc{A}$ has an axiomatization consisting of at most $2$-variable identities.
\end{enumerate}
Then $\cc{A}\circ\cc{B}$ is a variety.
\end{prop}

There is an unfortunate conflict of standard notation in this proof. For two binary relations, $\alpha$ and $\beta$ on a set, $A$, the relational product, $\alpha\circ\beta$ is the set $$\{(a,c):\text{there is $b\in A$ such that }(a,b)\in\alpha\text{ and }(b,c)\in\beta\}.$$

\begin{proof}
It is straightforward to show that $\cc{A}\circ\cc{B}$ is closed under taking products and subalgebras even if (1) and (2) do not hold, so we need only show that $\cc{A}\circ\cc{B}$ is closed under taking quotients. We begin with a claim that is inspired by Corollary~7.13 from \cite{hobbymckenzie}.
\begin{cla}
Let $\alga\in\cc{A}\circ\cc{B}$ and suppose $\theta\in\conn(\alga)$ witnesses this. For any $\alpha\in\conn(\alga)$, we have $\theta\circ\alpha\circ\theta\subseteq\alpha\circ\theta\circ\alpha$. Here, $\circ$ refers to the relational product.
\end{cla}
\begin{proof}[Proof of Claim]
Suppose $(a,d)\in\theta\circ\alpha\circ\theta$, which means there are $b,c\in A$ satisfying $(a,b)\in\theta,(b,c)\in\alpha$, and $(c,d)\in\theta$. Then
$$a=t(a,b)\overset{\alpha}{\equiv}t(a,c)\overset{\theta}{\equiv}t(b,d)\overset{\theta}{\equiv}t(d,b)\overset{\alpha}{\equiv}t(d,c)=d,$$
so $(a,d)\in\alpha\circ\theta\circ\alpha$.
\end{proof}
Since $\theta\join\alpha$ is transitive and contains both $\alpha$ and $\theta$, we have $\alpha\circ\theta\circ\alpha\subseteq\theta\join\alpha$. The opposite inclusion is implied by the Claim. Therefore, $\alpha\circ\theta\circ\alpha=\alpha\join\theta$.

We now wish to choose an arbitrary congruence, $\alpha\in\conn(\alga)$ and prove that $\alga/\alpha\in\cc{A}\circ\cc{B}$. The congruence $(\theta\join\alpha)/\alpha\in\conn(\alga/\alpha)$ has the property that $(\alga/\alpha)/((\theta\join\alpha)/\alpha)\cong \alga/(\theta\join\alpha)$, which is a homomorphic image of $\alga/\theta$, so it is in $\cc{B}$. To finish the proof, we will show that for any binary terms $u$ and $v$ such that $\cc{A}\vDash u(x,y)\approx v(x,y)$, we have that each $(\alpha\join\theta)/\alpha$-block satisfies $u(x,y)\approx v(x,y)$. This amounts to showing for any $(a,d)\in\theta\join\alpha$ that $u(a,d)\overset{\alpha}{\equiv}v(a,d)$. From the previous paragraph, there are $b,c\in A$ such that $(a,b)\in\alpha,(b,c)\in\theta$, and $(c,d)\in\alpha$. We then have
$$u(a,d)\overset{\alpha}{\equiv}u(b,c)=v(b,c)\overset{\alpha}{\equiv} v(a,d).$$
Since $\cc{A}$ has an axiomatization consisting of two variable identities, we have that each $(\alpha\join\theta)/\alpha$ block satisfies the identities in this axiomatization, so it is in $\cc{A}$.
\end{proof}

When similar idempotent varieties, $\cc{A}$ and $\cc{B}$ satisfy hypothesis (1) from Proposition~\ref{maltsevprodvariety}, and $\alga\in\cc{A}\circ\cc{B}$, there is a unique congruence on $\alga$ that witnesses it. The existence and definition of this congruence will be important in Section~\ref{mainresult}.

\begin{dfn}\label{specialcongruence}
Suppose $\cc{A}$ and $\cc{B}$ are similar idempotent varieties whose similarity type has a binary term, $t$, satisfying (1) from Proposition~\ref{maltsevprodvariety}. For $\alga\in\cc{A}\circ\cc{B}$, define $$\theta_{\alga} = \{(a,b)\in A^2:t(a,b)=a \text{ and } t(b,a)=b\}.$$
The subscript will be omitted whenever possible.
\end{dfn}

\begin{lem}\label{specialcongruencelem}
Let $\cc{A}$ and $\cc{B}$ be as in Definition~\ref{specialcongruence} and $\alga\in\cc{A}\circ\cc{B}$. Then $\theta_{\alga}\in\conn(\alga)$ and it is the unique congruence which witnesses $\alga\in\cc{A}\circ\cc{B}$.
\end{lem}

\begin{proof}
Using the assumption that $\alga\in\cc{A}\circ\cc{B}$, there is some congruence, $\alpha$, which witnesses $\alga\in\cc{A}\circ\cc{B}$. To prove the lemma, all we need to show is that $\theta=\alpha$. That $\alpha\subseteq\theta$ follows from the definition of $\theta$ and the fact that $t$ is the first projection on $\alpha$-blocks. If $(a,b)\in\theta$, then $a=t(a,b)\overset{\alpha}{\equiv}t(b,a)=b$, so $(a,b)\in\alpha$, which gives the other inclusion.
\end{proof}

For the next Lemma, $\cc{A}$ and $\cc{B}$ are similar varieties with a binary term, $\cdot$ which is a $2$-semilattice operation in $\cc{B}$, and the first projection in $\cc{A}$. By Lemma~\ref{specialcongruencelem}, any algebra, $\alga\in\cc{A}\circ\cc{B}$ has a unique congruence, $\theta_{\alga}$ witnessing $\alga\in\cc{A}\circ\cc{B}$. Furthermore, if $\alga$ is finite, then $\alga/\theta_{\alga}$ has a digraph structure defined in the same way as Definition~\ref{digraphdfn}. That is, for $a,b\in A$, $a/\theta\dig b/\theta$ when $a/\theta\cdot b/\theta=b/\theta$.

\begin{lem}\label{functionslem}
Suppose $\alga\in\cc{A}\circ\cc{B}$ is finite and satisfies $x\cdot(y\cdot z)\approx x\cdot(z\cdot y)$. For $a,b\in A$, if $a/\theta\dig b/\theta$, then there is a function $f_{a/\theta,b/\theta}:a/\theta\to b/\theta$ given by $f_{a/\theta,b/\theta}(x) = x\cdot b$. Moreover, this function is well defined in the sense that if $b\overset{\theta}{\equiv} b'$ then $f_{a/\theta,b/\theta}(x) = f_{a/\theta,b'/\theta}(x)$. 
\end{lem}

\begin{proof}
If $x\overset{\theta}{\equiv} a$, then $x\cdot b\overset{\theta}{\equiv} a\cdot b\overset{\theta}{\equiv} b$ where the second equivalence is because $a/\theta\dig b/\theta$. This shows that $f_{a/\theta,b/\theta}$ as defined is a function from $a/\theta\to b/\theta$. It remains to show that the function is well defined. Using the definition of $\theta$, if $b\overset{\theta}{\equiv}b'$, then
\begin{eqnarray*}
x\cdot b &=& x\cdot (b\cdot b') \\
&=& x\cdot(b'\cdot b) \\
&=& x\cdot b'.
\end{eqnarray*} 
\end{proof}

We will use Lemma~\ref{functionslem} to take advantage of Bulatov's proof that every nonempty \adj{} of $\csp(\algd)$ has a solution when $\algd$ has a $2$-semilattice operation. This will allow us to prove Theorem~\ref{mainthm} by first considering it as an instance where $\algd$ has a $2$-semilattice operation by taking quotients of each potato and relation. This solution is an assignment of congruence classes, and hence, subuniverses of the original potatoes. We can then check the induced subinstance on these congruence classes using the algorithm for $\cc{W}$. Lemma~\ref{functionslem} allows us to show that this is sufficient to determine whether or not the input \adj{} has a solution.

\section{Bulatov's work from \cite{bulatov1}}\label{bulatovsection}
In this section, we go through the proof of Bulatov's main result from \cite{bulatov1}. The first Lemma is one of Bulatov's observations, translated to our context.

\begin{lem}\label{23minstrong}
Let $\algd\in\cc{S}$ and suppose $\cc{I}=(X,\cc{P},\cc{R})$ is a nonempty \adj{} of $\csp(\algd)$. The instance, $\cc{I}'=(X,\{P_x':x\in X\},\{R_{xy}':(x,y)\in X^2\})$, is a nonempty \adj{}.
\end{lem}

\begin{proof}
Since $P_x$ and $R_{xy}$ are nonempty for each $x,y\in X$, we also have that $P_x'$ and $R_{xy}'$ are nonempty for all $x,y\in X$. We first show that $R_{xy}'\subseteq P_x'\times P_y'$. To see this, choose $(a,b)\in R_{xy}'$. We need to show that $a\in P_x'$ and $b\in P_y'$. We will show only the first since the second is similar to it. Pick $a'\in P_x$. Since $\cc{I}$ is a \adj{}, (P3) guarantees some $b'\in P_y$ with $(a',b')\in R_{xy}$. By Lemma~\ref{digraphproperties}~(3), there is a directed walk from $(a',b')$ to $(a,b)$ in $R_{xy}$. Restricting to the first coordinates, we get a walk in $P_x$ from $a'$ to $a$. Since $a'\in P_x$ was arbitrary, we have that $a\in P_x'$, again by Lemma~\ref{digraphproperties}~(3).

To see that $\cc{I}'$ satisfies (P1), Observe that (P1) for $\cc{I}$ implies that $\algp_x$ and $\algr_{xx}$ are isomorphic via the map given by $a\mapsto (a,a)$ for $a\in P_x$. This means $R_{xx}'$ will be precisely $\{(a,a):a\in P_x'\}$.
We now show that $\cc{I}'$ satisfies (P2). First, we define the algebra of triangles on $(x,y,z)$ by
$$T=\{(a,b,c)\in P_x\times P_y\times P_z:(a,b)\in R_{xy},(a,c)\in R_{xz},(b,c)\in R_{yz}\}$$
and show that its subuniverse,
$$S=\{(a,b,c)\in P_x'\times P_y'\times P_z':(a,b)\in R_{xy}',(a,c)\in R_{xz}',(b,c)\in R_{yz}'\},$$
is nonempty. If we take $(a_1,b_1)\in R_{xy}'$, by (P2) for $\cc{I}$, it extends to some $(a_1,b_1,c_1)\in T$. Similarly, there is $(a_2,b_2,c_2)\in T$ with $(a_2,c_2)\in R_{xz}'$, and $(a_3,b_3,c_3)\in T$ with $(b_3,c_3)\in R_{yz}'$. If we set $a=a_1\cdot (a_2\cdot a_3)$, $b=b_1\cdot(b_2\cdot b_3),$ and $c=c_1\cdot(c_2\cdot c_3)$, then an application of Lemma~\ref{digraphproperties}~(4) shows that $(a,b,c)$ is in $S$, so $S$ is nonempty. Now we take $(d,e)\in R_{xy}'$ and find $f\in P_z'$ such that $(d,e,f)\in S$. Again, we can use (P2) of $\cc{I}$ to find $f'\in P_z$ so that $(d,e,f')\in T$. Since $(d,e)\in R_{xy}'$, Lemma~\ref{digraphproperties}~(3) guarantees a walk in $R_{xy}$ from $(a,b)$ to $(d,e)$. Using this and (P2) of $\cc{I}$, we can find a sequence, $\{(u_i,v_i,w_i)\}_{i=1}^n$ of elements of $T$ so that
$$(a,b)\overset{\algr_{xy}}{\dig}(u_1,v_1)\overset{\algr_{xy}}{\dig}\cdots\overset{\algr_{xy}}{\dig}(u_n,v_n)\overset{\algr_{xy}}{\dig}(d,e).$$
Now define $w_1^*=c\cdot w_1$, and $w_i^*=w_{i-1}^*\cdot w_i$ for $2\leq i\leq n$. As well, set $f=w_n^*\cdot f'$. It follows from Lemma~\ref{digraphproperties}~(1) that
$$(a,b,c)\overset{\algt}{\dig}(u_1,v_1,w_1^*)\overset{\algt}{\dig}\cdots\overset{\algt}{\dig}(u_n,v_n,w_n^*)\overset{\algt}{\dig}(d,e,f).$$
Since each of $R_{xy}',R_{xz}'$, and $R_{yz}'$ is closed with respect to taking out neighbours, it follows that $S$ is as well. Therefore, since $(a,b,c)\in S$, the entire walk, including $(d,e,f)$, is in $S$. 
\end{proof}

Lemma~\ref{23minstrong} in the context of Bulatov's proof of Theorem~\ref{bulatovsresult} allows us to pass from an arbitrary \adj{} to one where every potato and relation is strongly connected. We can also get away with a different reduction when this one is not possible and some potato has more than one element. These two reductions together make up the bulk of the proof of Theorem~\ref{bulatovsresult}. Following \cite{bulatov1}, we now go through an in depth study of the subdirect relations between strongly connected potatoes. Using the theory of absorption developed by Barto and Kozik in \cite{bkcyclic}, we can simplify some of Bulatov's reasoning.

\begin{lem}\label{stronglyconnectedabsorptionfree}
Let $\alga\in\cc{S}$ be finite and strongly connected. Then $\alga$ has no proper absorbing subuniverse.
\end{lem}

\begin{proof}
First, we note that if $B\vartriangleleft\alga$, then there is some term, $t$, witnessing it that depends on all of its variables. If $u$ is an $n$-ary term witnessing $B\vartriangleleft\alga$ that depends on variables $x_1,\dots,x_k$, then $t(x_1,\dots,x_k)$ defined by $u(x_1,\dots,x_{k-1},x_k,x_k,\dots,x_k)$ has these properties.
Suppose $B\vartriangleleft \alga$ with respect to some $n$-ary term, $t$, which depends on all of its variables. Since $B$ is proper and $(A,\dig)$ is strongly connected, there are $b\in B$ and $c\in A-B$ with $b\dig c$. By Lemma~\ref{digraphproperties}~(5), the algebra $(\{b,c\};\cdot)$ is a semilattice with absorbing element $c$. Now consider $t(c,b,\dots,b)$. Since $t$ depends on $x_0$, it must mention $x_0$ syntactically. It follows that $t(c,b,\dots,b)=c$ since $c$ is the absorbing element of the semilattice. This is a contradiction, so there could have been no such $B$.
\end{proof}

We can now apply Theorem~\ref{absorptiontheorem} to narrow down the possibilities of the $R_{xy}$ in \adj s.

\begin{lem}\label{edgerelations}
Let $\alga,\algb\in\cc{S}$ be finite and strongly connected with $\algb$ simple. If $R\leq_{\sd}\alga\times\algb$ then either $R=A\times B$ or there is a surjective homomorphism $\varphi:\alga\to\algb$ such that $R=\{(a,\varphi(a)):a\in A\}$. In other words, $R$ is the graph of $\varphi$.
\end{lem}

\begin{proof}
If $R$ is the graph of a function, it is easily checked that it is necessarily the graph of a homomorphism. Therefore, if we assume $R$ is not the graph of a homomorphism, subdirectness of $R$ guarantees that there is some $a\in A$ and distinct $b_1,b_2\in B$ with both $(a,b_1)$ and $(a,b_2)\in R$. The relation
$$\tau=\{(c,d)\in B^2:\text{ there is }a\in A\text{ such that }(a,c),(a,d)\in R\}$$
is a symmetric and reflexive subuniverse of $\algb^2$, so its transitive closure, $\alpha$, is a congruence. We also have that $(b_1,b_2)\in\tau\subseteq\alpha$, so $\alpha=B\times B$ because $b_1\neq b_2$ and $\algb$ is simple. It follows that $R$ is linked. We also have that $\alga$ and $\algb$ are absorption free by Lemma~\ref{stronglyconnectedabsorptionfree}, and since $\cdot$ is commutative and idempotent, it is a Taylor operation for $\cc{S}$. The conditions of Theorem~\ref{absorptiontheorem} are satisfied, so $R=A\times B$.
\end{proof}

\begin{lem}\label{simplescrelations}
Let $\alga,\algb\in\cc{S}$ be finite, strongly connected, and simple with $R\leq_{\sd}\alga\times\algb$. Either $R=A\times B$ or $R$ is the graph of a bijection.
\end{lem}

\begin{proof}
By Lemma~\ref{edgerelations}, if $R\neq A\times B$, then it is both the graph of a surjective homomorphism from $\alga$ to $\algb$ and vice versa. Since $A$ and $B$ are finite, this means $|A|=|B|$, so each of these surjections is a bijection.
\end{proof}

The next Lemma is Lemma~3.8 from \cite{bulatov1}. The proof is technical and can be found there.

\begin{lem}[Lemma~3.8 from \cite{bulatov1}]\label{bul3.8}
Let $\alga_1,\alga_2,\alga_3\in\cc{S}$ be finite and strongly connected. Suppose $T\leqsd\alga_1\times\alga_2\times\alga_3$ satisfies the following:
\begin{enumerate}
\item $\alga_3$ is simple,
\item $\pr_{1,2}(T)$ is strongly connected,
\item $\pr_{i,3}(T) = A_i\times A_3$ for $i=1,2$.
\end{enumerate}
Then $T=\pr_{1,2}(T)\times A_3$. 
\end{lem}

Definition~\ref{decompdfn} is very similar to Definition~8.2 of a {\it decomposition} from \cite{bkbounded}. The idea is that when we have an instance in which every potato and constraint is strongly connected, if there is a potato with more than one element, we can decompose the instance into one instance for each class of a maximal congruence on that potato.

\begin{dfn}\label{decompdfn}
Let $\cc{I}$ be a \adj{} of $\csp(\algd)$ for some $\algd\in\cc{S}$ in which every potato and relation is strongly connected. Further suppose that there is some $u\in X$ such that $|P_u|>1$. Choose a maximal congruence, $\alpha_u$ of $\algp_u$. Let $W\subseteq X$ be the set of variables such that $\{(a,b/\alpha_u):(a,b)\in\algr_{xu}\}$ is the graph of a surjective homomorphism from $\algp_x$ to $\algp_u/\alpha_u$, and for each $x\in W$, let $\varphi_{xu}$ be this surjection. Now let $P_u^1,P_u^2,\dots,P_u^k$ be the $\alpha_u$ classes in $\algp_u$, and define instances $\cc{I}_1,\dots,\cc{I}_k$, each with variable set $X$ by
\begin{enumerate}
\item[-] $P_x^{\cc{I}_i} = \varphi_{xu}^{-1}(P_u^i)$ if $x\in W$, and $P_x$, otherwise.
\item[-] $R_{xy}^{\cc{I}_i} = R_{xy}\cap(P_x^{\cc{I}_i}\times P_y^{\cc{I}_i})$
\end{enumerate}
\end{dfn}

Since $R_{uu}=0_{P_u}$, we have that $u\in W$ and $P_u^{\cc{I}_i}=P_u^i$. We now state an easy Lemma based on Definition~\ref{decompdfn}.

\begin{lem}\label{itsadecomp}\hspace{1cm}
\begin{enumerate}
\item With $W$ as in Defintion~\ref{decompdfn}, if $x,\in W$, then $R_{xy}\cap (P_x^{\cc{I}_i}\times P_y)=R_{xy}\cap (P_x^{\cc{I}_i}\times P_y^{\cc{I}_i})=R_{xy}^{\cc{I}_i}$ for every $y\in X$. The second equality is by definition.
\item Let $\cc{I}$ be a \adj{} of $\csp(\algd)$ for $\algd\in\cc{S}$ with strongly connected potatoes and relations. Suppose $|P_u|>1$ and $\cc{I}_1,\dots,\cc{I}_k$ is the decomposition from Definition~\ref{decompdfn}. If $s$ is a solution to $\cc{I}$, then it is a solution to $\cc{I}_i$ for some $i$.
\item For $x\in W$, there is a congruence, $\alpha_x$ on $\algp_x$ such that $\algp_x/\alpha_x\cong\algp_u/\alpha_u$ via the isomorphism $P_x^{\cc{I}_i}\mapsto P_u^{\cc{I}_i}$.
\end{enumerate}
\end{lem}

\begin{proof}
For (1), the left to right inclusion is the interesting one. To see that it is true, we first note that if $y\notin W$, then $P_y^{\cc{I}_i}=P_y$, so there is nothing to prove. From now on, we assume that $y\in W$ as well. Suppose $(a,b)\in R_{xy}\cap (P_x^{\cc{I}_i}\times P_y)$. By (P2) of $\cc{I}$, we get some $c\in P_u$ such that $(a,c)\in R_{xu}$ and $(b,c)\in R_{yu}$. Since $x\in W$ and $a\in P_x^{\cc{I}_i}$, we have that $c\in P_u^{\cc{I}_i}$. This is because $(a,c)\in R_{xu}$ means $c\in \varphi_{xu}(a)=P_u^{\cc{I}_i}$. Since $y\in W$ and $(b,c)\in R_{yu}$, the same reasoning in reverse puts $b\in\varphi_{yu}^{-1}(P_u^{\cc{I}_i})= P_y^{\cc{I}_i}$. As for (2), because of the way the $R_{xy}^{\cc{I}_i}$ are defined, it suffices to show that there is some $i$ such that $s(x)\in P_x^{\cc{I}_i}$ for all $x\in X$. Since the $P_u^{\cc{I}_1},\dots,P_u^{\cc{I}_k}$ partition $P_u$, there is some $i$ such that $s(u)\in P_u^{{\cc{I}_i}}$. This $i$ does the trick. For $x\in W$, since $(s(x),s(u))\in R_{xu}$, we have that $s(x)\in\varphi_{xu}^{-1}(P_u^{\cc{I}_i})=P_x^{\cc{I}_i}$, and for $x\notin W$, $s(x)\in P_x^{\cc{I}_i}$ since it is all of $P_x$. Part (3) is simply because the $P_x^{\cc{I}_i}$ are chosen to be the classes of the kernel of $\varphi_{xu}:P_x\to P_u/\alpha_u$.
\end{proof}

\begin{lem}\label{nonsimplepotato}
Let $\cc{I}$ be a \adj{} of $\csp(\algd)$ for some $\algd\in \cc{S}$ in which each potato and relation is strongly connected and $u\in X$ is such that $|P_u|>1$. Choose a maximal congruence $\alpha_u$ on $\algp_u$ and construct $\cc{I}_1,\dots,\cc{I}_k$ as in Definition~\ref{decompdfn}. Then each $\cc{I}_i$ is \adj.
\end{lem}

\begin{proof}
For any $x\in W$, let $\alpha_x$ be the kernel of $\varphi_{xu}$ from Definition~\ref{decompdfn}. We first show that for any $x\in W$ and $y\notin W$ that $R_{yx}^{\alpha_x}:=\{(a,b/\alpha_x):(a,b)\in R_{yx}\}=P_y\times P_x/\alpha_x$. Since $R_{yx}\leq_{\sd}\algp_y\times\algp_x$, we get that $R_{yx}^{\alpha_x}\leq_{\sd}\algp_y\times\algp_x/\alpha_x$. By Lemma~\ref{itsadecomp}~(3), $\algp_x/\alpha_x\cong \algp_u/\alpha_u$ and since $\alpha_u$ is maximal, $\algp_x/\alpha_x$ is simple. It then follows from Lemma~\ref{edgerelations} that $R_{yx}^{\alpha_x}$ is either the graph of a surjective homomorphism or the full product. We will show that if it is the graph of a surjective homomorphism, then $y\in W$. Assume $R_{yx}^{\alpha_x}$ is the graph of a surjective homomorphism from $\algp_y$ to $\algp_x/\alpha_x$ and assume $(a,c_1),(a,c_2)\in R_{yu}$. Since (P2) holds in $\cc{I}$, there are $b_1,b_2\in P_x$ with $(a,b_1),(a,b_2)\in R_{yx}$ and $(b_1,c_1),(b_2,c_2)\in R_{xu}$. Since $R_{yx}^{\alpha_x}$ is the graph of a surjective homomorphism, the former implies $b_1/\alpha_x=b_2/\alpha_x$. From this and the latter, we get that $c_1/\alpha_u=c_2/\alpha_u$. This means $\{(a,b/\alpha_u):(a,b)\in R_{yu}\}$ is the graph of a surjective homomorphism, so $y\in W$. Therefore, if $y\notin W$ and $x\in W$, $R_{yx}^{\alpha_x}=P_y\times P_x/\alpha_x$, as desired. We will now show that each $\cc{I}_i$ is a \adj{}.

By definition, $R_{xy}^{\cc{I}_i}\subseteq P_x^{\cc{I}_i}\times P_y^{\cc{I}_i}$. Each $\cc{I}_i$ also satisfies (P1) because of the way the $R_{xy}^{\cc{I}_i}$ were defined. We now fix and $i$ and show that $\cc{I}_i$ satisfies (P2). This means we need to show that for any $z\in X$ and $(a,b)\in R_{xy}^{\cc{I}_i}$, there is $c\in P_z^{\cc{I}_i}$ such that $(a,c)\in R_{xz}^{\cc{I}_i}$ and $(b,c)\in R_{yz}^{\cc{I}_i}$. We do know that there is some $c\in P_z$ such that $(a,c)\in R_{xz}$ and $(b,c)\in R_{yz}$. If $z\notin W$, then $P_z^{\cc{I}_i}=P_z$ and there is nothing to show. Therefore, from now on, we assume $z\in W$. First, suppose $x\in W$ and $(a,b)\in R_{xy}^{\cc{I}_i}$. (P2) for $\cc{I}$ provides $c\in P_z$ such that $(a,c)\in R_{xz}$ and $(b,c)\in R_{yz}$. Since $(a,b)\in R_{xy}^{\cc{I}_i}$, we know that $a\in P_x^{\cc{I}_i}$, so $(a,c)\in R_{xz}\cap(P_x^{\cc{I}_i}\times P_z)$ which equals $R_{xz}^{\cc{I}_i}$ by Lemma~\ref{itsadecomp}~(1), so $c\in P_z^{\cc{I}_i}$. The case when $y\in W$ is similar, so the remaining case is when $x,y\notin W$. First we set
$$T=\{(a,b,c/\alpha_z):(a,b)\in R_{xy},(a,c)\in R_{xz},\text{ and }(b,c)\in R_{yz}\}.$$
Now we set $\alga_1=\algp_x,\alga_2=\algp_y$, and $\alga_3=\algp_z/\alpha_z$. Since $\alga_3\cong \algp_u/\alpha_u$, $\alga_3$ is simple. Since (P2) holds for $\cc{I}$ and $x,y\notin W$, we have $\pr_{1,2}(T)=R_{xy}$, which is strongly connected by assumption. By the first paragraph of this proof and the fact that $\cc{I}$ has (P2), we also get that $\pr_{i,3}=A_i\times A_3$ for $i=1,2$. The conditions of Lemma~\ref{bul3.8} are satisfied, so $T=R_{xy}\times (P_z/\alpha_z)$. Therefore, $(a,b,P_z^{\cc{I}_i})\in T$, so there is some $c\in P_z^i$ with $(a,c)\in R_{xz}$ and $(b,c)\in R_{yz}$. Since $a\in P_x^{\cc{I}_i}$, $b\in P_y^{\cc{I}_i}$, and $c\in P_z^{\cc{I}_i}$, we actually have $(a,c)\in R_{xz}^{\cc{I}_i}$ and $(b,c)\in R_{yz}^{\cc{I}_i}$. This completes the proof that $\cc{I}_i$ satisfies (P2).

\end{proof}

To finish off this section, we will give a proof of Bulatov's result from \cite{bulatov1}. The next definition gives notation to turn Bulatov's proof into a constructive one.

\begin{dfn}\label{inequalitiesdfn}
Let $\cc{I}$ be a \adj{} of $\csp(\algd)$ for some $\algd\in\cc{S}$, and suppose $\cc{J}$ is a subinstance of $\cc{I}$.
\begin{enumerate}
\item We write $\cc{I}\geqstrong\cc{J}$ if $P_x^{\cc{J}}=(P_x^{\cc{I}})'$ for each $x$, and $R_{xy}^{\cc{J}}=(R_{xy}^{\cc{I}})'$ for each $(x,y)\in X^2$.
\item We write $\cc{I}\geqcong\cc{J}$ if $\cc{J}$ is one of the subinstances $\cc{I}_1,\dots,\cc{I}_k$ from Definition~\ref{decompdfn}. We are implicitly assuming that every potato and relation in $\cc{I}$ is strongly connected, and for some $u\in X$, $|P_u|>1$.
\end{enumerate}
\end{dfn}

Here is the main result of \cite{bulatov1}, rephrased in our language.

\begin{thm}[Theorem 3.1 in \cite{bulatov1}]\label{bulatovsresult}
Let $\cc{I}$ be a \adj{} of $\csp(\algd)$ where $\algd\in\cc{S}$ and all potatoes are nonempty. Then $\cc{I}$ has a solution.
\end{thm}

\begin{proof}
Since all of the potatoes in $\cc{I}$ are finite, we can use Lemmas~\ref{23minstrong}~and~\ref{nonsimplepotato} repeatedly to construct a sequence of \adj s
$$\cc{I}=\cc{I}_0\geqstrong\cc{I}_1\geqcong\cc{I}_2\geqstrong\cdots\geqcong\cc{I}_{n-1}\geqstrong\cc{I}_n$$
where the potatoes, and hence, relations in $\cc{I}_n$ each have one element. A \adj{} whose potatoes are all singletons has a solution. In fact, it has exactly one solution.
\end{proof}

\section{Bulatov Solutions}\label{bulatovsolutionssection}
Bulatov solutions, the subject of Definition~\ref{bulatovsolutions}, are the special solutions we seek that were alluded to at the end of Section~\ref{varsection}. They are solutions to an instance like $\cc{I}_n$ occurring at the end of a chain as in the proof of Theorem~\ref{bulatovsresult}.

\begin{lem}\label{inequalitiesarrows}
Suppose $\cc{I}$ is a \adj{} of $\csp(\algd)$ with $\algd\in\cc{S}$ and $\cc{J}$ is a subinstance of $\cc{I}$.
\begin{enumerate}
\item If $\cc{I}\geqstrong \cc{J}$ and $s$ is a solution to $\cc{I}$, there is a solution $r$ to $\cc{J}$ with $s\dig r$.
\item Suppose every potato and relation of $\cc{I}$ is strongly connected, $|P_u|>1$, and $\cc{I}_1,\dots,\cc{I}_k$ are the instances from Definition~\ref{decompdfn}. Further, suppose $\cc{J}=\cc{I}_i$ for some $i$. In other words, $\cc{I}\geqcong\cc{J}$. If $s$ is a solution to $\cc{I}$, there is a directed walk from $s$ to some solution, $r_i$ of $\cc{J}=\cc{I}_i$.
\end{enumerate} 
\end{lem}

Since the algebra of solutions is in the variety generated by $\algd$, it is a $2$-semilattice, so it makes sense to say $s_1\dig s_2$ for solutions $s_1$ and $s_2$. We are working towards showing that if $\cc{I}$ is a \adj{} with a solution, $s$ and $\cc{I}_n$ is the \adj{} from the proof of Theorem~\ref{bulatovsresult}, there is a directed walk from $s$ to the solution of $\cc{I}_n$. This will be formulated more precisely in Corollary~\ref{solutionspushdown}.

\begin{proof}
We prove (1) first. Suppose $\cc{I}\geqstrong\cc{J}$. We know from Lemma~\ref{23minstrong} that $\cc{J}$ is a \adj{}, so by Theorem~\ref{bulatovsresult}, it has a solution, $r'$. The solution $r=s\cdot r'$ has the property that $s\dig r$ by Lemma~\ref{digraphproperties}~(1). It can be seen using Lemma~\ref{digraphproperties}~(4) that $r$ is a solution to $\cc{J}$.

The proof of (2) is not quite as simple. First, by Lemma~\ref{itsadecomp}~(2), we have that $s$ is a solution to some $\cc{I}_j$. Since we can order the instances any way we like, we'll assume $s$ is a solution to $\cc{I}_1$. Now suppose, for some $j$, that $P_u^{\cc{I}_1}\cdot P_u^{\cc{I}_j}=P_u^{\cc{I}_j}$ (as elements of $P_u/\alpha_u$ from Definition~\ref{decompdfn}). Let $r_j=s\cdot r_j'$ where $r_j'$ is a solution to $\cc{I}_j$ which exists by Theorem~\ref{bulatovsresult}. By Lemma~\ref{digraphproperties}~(1), $s\dig r_j$. We now show that $r_j$ is a solution to $\cc{I}_j$. By the assumption on $j$, $P_u^{\cc{I}_1}\dig P_u^{\cc{I}_j}$ in $\algp_u/\alpha_u$, so by Lemma~\ref{itsadecomp}~(3), $P_x^{\cc{I}_1}\dig P_x^{\cc{I}_j}$ as well for any $x\in W$. It follows that for $x\in W$ that $r_j(x)\in P_x^{\cc{I}_j}$. Since $P_x^{\cc{I}_j}=P_x$ for $x\notin W$, we get that $r_j(x)\in P_x^{\cc{I}_j}$ for all $x\in X$. This shows that $r_j$ is, a solution to $\cc{I}_j$. To finish off the proof, note that $\algp_u/\alpha_u$ is strongly connected by Lemma~\ref{digraphproperties}~(6). Therefore, there is a directed walk, $P_u^{\cc{I}_1}\dig P_u^{\cc{I}_{k_1}}\dig\dots\dig P_u^{\cc{I}_{k_m}}=P_u^{\cc{I}_i}$. Each $\cc{I}_{k_{\ell}}$ has a solution, so the previous argument can be repeated to obtain a directed walk from $s$ to a solution through $\cc{I}_i=\cc{J}$.
\end{proof}

We now define the set of Bulatov solutions to a \adj{} of $\csp(\algd)$ where $\algd\in\cc{S}$. They are of importance in the proof of the main result of this paper.

\begin{dfn}\label{bulatovsolutions}
Let $\cc{I}$ be a \adj{} of $\csp(\algd)$ with $\algd\in\cc{S}$. A Bulatov solution to $\cc{I}$ is a solution, $r$ that arises as a solution to some \adj{} $\cc{I}_n$ whose potatoes are all singletons occurring at the end of a sequence of instances, $\cc{I}_0,\dots,\cc{I}_n$ with
$$\cc{I}=\cc{I}_0\geqstrong\cc{I}_1\geqcong\cc{I}_2\geqstrong\cdots\geqcong\cc{I}_{n-1}\geqstrong\cc{I}_n.$$
\end{dfn}

\begin{cor}\label{solutionspushdown}
Let $\cc{I}$ be a \adj{} of $\csp(\algd)$ with $\algd\in\cc{S}$. Suppose $s$ is a solution to $\cc{I}$, and $r$ is a Bulatov solution to $\cc{I}$. Then there is a directed walk from $s$ to $r$ in the algebra of solutions to $\cc{I}$.
\end{cor}

\begin{proof}
By repeatedly applying Lemma~\ref{inequalitiesarrows} and concatenating the walks each application provides, we get a directed walk from $s$ to the unique solution of $\cc{I}_n$.
\end{proof}

Bulatov's work shows that a \adj{} of $\csp(\algd)$ with nonempty potatoes has a solution when $\algd\in\cc{S}$, but we have now seen that Bulatov's proof can be followed carefully to find a Bulatov solution to any such \adj{}. This is made precise in Algorithm~\ref{bulatovsolutionalgorithm}.

\begin{algorithm}
\caption{Find a Bulatov solution to a nonempty \adj{} of $\csp(\algd)$ when $\algd\in\cc{S}$.}\label{bulatovsolutionalgorithm}
\begin{algorithmic}[1]
\State {\bf Input}: A nonempty \adj{} of $\csp(\algd)$ with $\algd\subseteq\cc{S}$.
\For {$x\in X$}
	\If {$(P_x,\overset{\algp_x}{\dig})$ is not strongly connected}
	\State $P_x \gets P_x'$
	\EndIf
\EndFor
\For {$(x,y)\in X^2$}
	\If {$(R_{xy},\overset{\algr_{xy}}{\dig})$ is not strongly connected}
		\State $R_{xy} \gets R_{xy}'$
	\EndIf
\EndFor
\If {$|P_u|>1$ for some $u\in X$}
	\State Choose $\alpha_u$ maximal in $\conn(\algp_u)$ and find $W$ (Definition~\ref{decompdfn})
	\For {$x\in X$}
		\If {$x\in W$}
			\State $P_x \gets \varphi_{xu}^{-1}(P_u^1)$
		\EndIf
	\EndFor
	\For {$(x,y)\in X^2$}
		\State $R_{xy} \gets R_{xy}\cap P_x\times P_y$
	\EndFor
	\State Go to 2
\Else
	\State $r(x)=a$ where $a$ is the unique element in $P_x$
\EndIf
\State {\bf Output}: $r$
\end{algorithmic}
\end{algorithm}

\section{A proof of Theorem~\ref{mainthm}}\label{mainresult}
For the duration of this section, we fix a variety, $\cc{W}$, and a finite algebra $\algd$, satisfying the properties in the statement of the Theorem. The variety generated by $\algd/\theta$ will be called $\cc{T}$.

By definition, $\algd\in\cc{W}\circ\cc{T}$. Since Maltsev products are closed under taking subalgebras and products, we get that every subalgebra of $\algd$ and $\algd\times\algd$ is in $\cc{W}\circ\cc{T}$. By Lemma~\ref{specialcongruencelem}, each such algebra, $\alga$, has a unique congruence, $\theta_{\alga}$ which witnesses $\alga\in\cc{W}\circ\cc{T}$.

\begin{dfn}\label{quotientinstance}
Let $\cc{I}=(X,\cc{P},\cc{R})$ be a \adj{} of $\csp(\algd)$. The quotient instance $\cc{I}/\theta$ is the instance $(X,\cc{P}/\theta,\cc{R}/\theta)$ where $\cc{P}/\theta=\{\algp_x/\theta_{\algp_x}:x\in X\}$ and $\cc{R}/\theta=\{\algr_{xy}^{\theta}:x,y\in X\}$ where $\algr_{xy}^{\theta}=\{(a/\theta_{P_x},b/\theta_{P_y}):(a,b)\in R_{xy}\}$.
\end{dfn}

We now acknowledge that the instance $\cc{I}/\theta$ constructed in Definition~\ref{quotientinstance} is not technically an instance of $\csp(\alge)$ for any algebra, $\alga$. This is because there is no reason to expect that there is any algebra, $\alge$ such that $\algp_x/\theta_x\leq\alge$ for every $x\in X$. We work around this technicality by thinking of the algebras as subuniverses of $\alge=\prod\{A:\varnothing\neq A\leq D\}$. Lemma~\ref{quotient23minimal} formalizes the desired properties of this algebra. Since $\algd$ is fixed, we take all of its subuniverses as part of an input to $\csp(\algd)$, which means this will not affect whether or not Algorithm~\ref{mainalgorithm} runs in polynomial time.

\begin{lem}\label{quotient23minimal}
If $\cc{I}$ is a \adj{} of $\csp(\algd)$. There is an algebra, $\alge\in\cc{T}$ depending only on $\algd$ with subalgebras, $(\algq_x:x\in X)$, and subalgebras of its square, $(S_{xy}:(x,y)\in X^2)$, such that
\begin{enumerate}
\item $\algp_x/\theta_{\algp_x}\cong \algq_x$ via the isomorphism $h_x$.
\item The map $h_{xy}:R_{xy}\to S_{xy}$ given by $(a,b)\mapsto (h_x(a),h_y(b))$ is an isomorphism for each $x,y\in X$.
\item $(X,(\algq_x:x\in X),(S_{xy}:(x,y)\in X^2))$ is a \adj{} of $\csp(\alge)$.
\end{enumerate}
\end{lem}

Because of this Lemma, we can think of $\cc{I}/\theta$ as a \adj{} of $\csp(\alge)$ for some algebra, $\alge\in\cc{T}$.

\begin{dfn}
Suppose $\cc{I}$ is a \adj{} of $\csp(\algd)$. Let $s$ be a solution to $\cc{I}$ and $\varphi$ be a solution to $\cc{I}/\theta$. We say that $s$ passes through $\varphi$ if for each $x$, $s(x)/\theta_{P_x}=\varphi(x)$.
\end{dfn}

If $s$ is a solution to $\cc{I}$, then the map given by $\varphi(x)=s(x)/\theta_{P_x}$ is the unique solution to $\cc{I}/\theta$ through which $s$ passes. Since $\cdot$ is a $2$-semilattice operation on each $\algp_x/\theta_x$, we can define a digraph relation by $a/\theta\dig b/\theta$ when $a\cdot b\overset{\theta}{\equiv} b$ as we did for $2$-semilattices. This also applies to the set of solutions to $\cc{I}/\theta$.

\begin{lem}\label{solutionsmoveacrossarrows}
Let $\cc{I}$ be a \adj{} of $\csp(\algd)$. Suppose $\varphi$ and $\psi$ are solutions to $\cc{I}/\theta$ with $\varphi\dig\psi$. If $\cc{I}$ has a solution through $\varphi$, then it has a solution through $\psi$.
\end{lem}

\begin{proof}
Since $\varphi\dig\psi$, for each $x$, we have $\varphi(x)\dig\psi(x)$ in $P_x/\theta$. By Lemma~\ref{functionslem}, there is a well defined function, $f_{\varphi(x),\psi(x)}:\varphi(x)\to\psi(x)$. Let $s$ be a solution passing through $\varphi$. The map $t:X\to \bigcup\{P_x:x\in X\}$ given by $t(x) = f_{\varphi(x),\psi(x)}(s(x))$ is a solution to $\cc{I}$ passing through $\psi$. That it passes through $\psi$ is simply because $f_{\varphi(x),\psi(x)}:\varphi(x)\to\psi(x)$ and $s(x)\in\varphi(x)$. To see that it is a solution, we must show, for any $x,y\in X$, that $(t(x),t(y))\in R_{xy}$. Since $\psi$ is a solution to $\cc{I}/\theta$, we have $(\psi(x),\psi(y))\in R_{xy}^{\theta}$, which means there is some $(a,b)\in R_{xy}$ such that $a/\theta_{P_x}=\psi(x)$ and $b/\theta_{P_y}=\psi(y)$. By Lemma~\ref{functionslem}, we have that $f_{\varphi(x),\psi(x)}(s(x))=s(x)\cdot a$, and $f_{\varphi(y),\psi(y)}(s(y))=s(y)\cdot b$. Therefore, $(t(x),t(y))=(s(x)\cdot a,s(y)\cdot b)\in R_{xy}$ because $((s(x),s(y))$ and $(a,b)$ are in $R_{xy}$.
\end{proof}

We now state the algorithm for $\csp(\algd)$.

\begin{algorithm}
\caption{Given a nonempty \adj{} of $\csp(\algd)$, determine whether or not it has a solution.}\label{mainalgorithm}
\begin{algorithmic}[1]
\State {\bf Input}: A nonempty \adj{} of $\csp(\algd)$.
\State Construct $\cc{I}/\theta$ from Definition~\ref{quotientinstance}.
\State Using Algorithm~\ref{bulatovsolutionalgorithm}, find a Bulatov solution, $\varphi$ to $\cc{I}/\theta$.
\For {$x\in X$}
\State $P_x \gets \varphi(x)$
\EndFor
\For {$(x,y)\in X^2$}
\State $R_{xy}\gets R_{xy}\cap (P_x\times P_y)$
\EndFor
\State Run the algorithm for $\cc{W}$ on $\cc{I}$.
\If {$\cc{I}$ has a solution}
\State {\bf Output}: YES
\Else
\State {\bf Output}: NO
\EndIf
\end{algorithmic}
\end{algorithm}

Note that in line 3, we have said to use Algorithm~\ref{bulatovsolutionalgorithm} on $\cc{I}/\theta$ when it is technically inappropriate. By Lemma~\ref{quotient23minimal}, we can consider $\cc{I}/\theta$ as a \adj{} of $\csp(\alge)$ for some $\alge\in\cc{T}$, but Algorithm~\ref{bulatovsolutionalgorithm} only works if $\alge\in\cc{S}$. However, since $\cc{T}$ has a $2$-semilattice operation, we can simply ignore all other operations without changing the solutions.

\begin{proof}[Proof of Theorem~\ref{mainthm}]
It suffices to show that Algorithm~\ref{mainalgorithm} correctly decides whether or not a \adj{} of $\csp(\algd)$ has a solution. Since line 8 checks a subinstance of the input instance for solutions, if the output of the algorithm is YES, there is a solution. We must show that if there is a solution, the algorithm will output YES. This amounts to showing that if $\cc{I}$ has a solution, then $\cc{I}$ has a solution through every Bulatov solution to $\cc{I}/\theta$. Suppose $\cc{I}$ has a solution, $s$, and let $\varphi$ be the solution to $\cc{I}/\theta$ through which $s$ passes. Fix a Bulatov solution to $\cc{I}/\theta$. By Corollary~\ref{solutionspushdown}, there is a directed walk in the algebra of solutions to $\cc{I}/\theta$ from $\varphi$ to $\psi$. Repeated application so Lemma~\ref{solutionsmoveacrossarrows} shows that $\cc{I}$ has a solution which passes through $\psi$.
\end{proof}

\section{A proof of Corollary~\ref{maincor}}\label{applicationsection}
We now apply Theorem~\ref{mainthm} to prove Corollary~\ref{maincor}. Since it is about edge term varieties, we start by giving the definition of an edge term.

\begin{dfn}
A $k$-edge operation is a $(k+1)$-ary operation satisfying the following identities:
\begin{eqnarray*}
e(y,y,x,x,x,x,\dots ,x,x) & \approx & x \\
e(y,x,y,x,x,x,\dots ,x,x) & \approx & x \\
e(x,x,x,y,x,x,\dots ,x,x) & \approx & x \\
e(x,x,x,x,y,x,\dots ,x,x) & \approx & x \\
\vdots \hspace{1.75cm} & \vdots & \hspace{.4mm}\vdots \\
e(x,x,x,x,x,x,\dots ,y,x) & \approx & x \\
e(x,x,x,x,x,x,\dots ,x,y) & \approx & x.
\end{eqnarray*}
An idempotent variety, $\cc{W}$, is called an edge term variety if its type has a $(k+1)$-ary term which is a $k$-edge term operation for every algebra in $\cc{W}$. 
\end{dfn}

\begin{lem}\label{edgetermlem}
Suppose $\cc{T}$ is a variety which is term equivalent to $\cc{S}$, and $\cc{W}$ is an edge term variety of the same type as $\cc{T}$. There is a binary term, $\cdot$, in the type of the two varieties which is a $2$-semilattice operation for $\cc{T}$ and the first projection in $\cc{W}$.
\end{lem}

\begin{proof}
Suppose $e$ is the term which is an edge term for $\cc{W}$, and $*$ is a $2$-semilattice operation for $\cc{T}$. Let $I=\{k:e\text{ depends on $x_k$ in $\cc{T}$}\}$. If $I\subseteq\{1,2\}$, then $x\cdot y=e(x*y,x*y,x,\dots,x)$ is the first projection in $\cc{W}$ and $x\cdot y\approx x*y$ in $\cc{T}$because $\cc{T}$ is idempotent. If $I\subseteq\{1,3\}$, the term $x\cdot y=e(x*y,x,x*y,x,\dots,x)$ satisfies the conditions. If $I=\{i\}$ for some $i\geq 4$, then $x\cdot y=e(x,x,x,\dots,x,x*y,x,\dots,x)$ satisfies the condition because $\cc{T}$ is idempotent. The $x*y$ occurs in the $i$-th position.

Now suppose $I=\{i,j\}$ with $i<j$. We have already seen that the result holds when $i=1$ and $j=2$ or $j=3$. Otherwise, $j\geq 3$ and there is some identity $e(u_1,\dots,u_{k+1})\approx x$ which holds in $\cc{W}$ where the $u_n\in\{x,y\}$, $u_i=x$, and $u_j=y$. In this case, we take $x\cdot y=e(u_1,\dots,u_{k+1})$. This term is the first projection in $\cc{W}$, and since $\cc{T}$ is term equivalent to the variety of $2$-semilattices, we have that $\cc{T}\vDash x\cdot y\approx x*y$ by Corollary~\ref{termequivcor}. A similar argument works if $|I|\geq 3$. For example, if $I=\{1,3,5\}$, we take $x\cdot y=e(y,x,y,x,x,\dots,x)$. This is the first projection in $\cc{W}$, and, since $e$ depends on variables $1,3$, and $5$ in $\cc{T}$, $\cc{T}\vDash x\cdot y\approx x*y$ for the same reason as before.
\end{proof}

\begin{proof}[Proof of Corollary~\ref{maincor}]
Fix an edge term variety, $\cc{W}$, and a similar variety, $\cc{T}$ which is term equivalent to $\cc{S}$. First, we define $\cc{W}^{(2)}$ to be the variety, similar to $\cc{W}$, axiomatized by the at most two variable identities which hold in $\cc{W}$. Since it has fewer identities, we get that $\cc{W}\leq\cc{W}^{(2)}$ and since edge terms are axiomatized by two variable identities, $\cc{W}^{(2)}$ is also an edge term variety. Therefore, it suffices to prove the Corollary in the case where $\cc{W}$ has an axiomatization consisting of only at most two variable identities. By Lemma~\ref{edgetermlem}, there is a binary term, $\cdot$, which is a $2$-semialttice operation in $\cc{T}$ and the first projection in $\cc{W}^{(2)}$. It was mentioned in the introduction that varieties having an edge term are tractable, so Theorem~\ref{mainthm} implies that every finite $\algd$ in $\cc{W}\circ\cc{T}$ satisfying $x\cdot(y\cdot z)\approx x\cdot(z\cdot y)$ is tractable. By Proposition~\ref{maltsevprodvariety}, $\cc{W}\circ\cc{T}$ is a variety, so the class, $\cc{U}$, of its members satisfying this identity is also a variety. By the properties of $\cdot$, each of $\cc{W}$ and $\cc{T}$ satisfy this identity, and it is not hard to see that both varieties are contained in their Maltsev product. It follows that $\cc{W}\join\cc{T}\leq\cc{U}$, so all finite members of $\cc{W}\join\cc{T}$ are tractable. 
\end{proof}

\section{Concluding remarks}
The defining identity, $x\cdot(y\cdot z)\approx x\cdot(z\cdot y)$, of $\cc{U}$ is what makes Lemma~\ref{functionslem} work. This Lemma is critical in the proof of correctness of Algorithm~\ref{mainalgorithm}. We construct an example of an algebra, $\algd$, and a \adj{} of $\csp(\algd)$ to demonstrate this. The type has a single ternary symbol, $q$. Let $\cc{W}$ be the variety axiomatized by $q(x,y,y)\approx q(y,x,y)\approx q(y,y,x)\approx x$. An edge term for $\cc{W}$ is $q$. In $\cc{T}$, $q$ is totally symmetric, and $x\cdot y$ defined by $q(x,y,y)$ is a semilattice operation. In other words, it is idempotent, commutative, and associative. More generally, it is a $2$-semilattice operation. Let $D=\{\top,(0,0),(0,1),(1,0),(1,1)\}$ and define
$$q^{\algd}(a_1,a_2,a_3)=\left\{\begin{array}{cl} \top & \text{if $a_1=a_2=a_3=\top$} \\ a_1+a_2+a_3 & \text{if $a_1,a_2,a_3\in\bb{Z}_2\times\bb{Z}_2$} \\ a_i & \text{if neither of the first two cases hold and} \\ & \text{$i$ is the smallest for which $a_i\in\bb{Z}_2\times\bb{Z}_2$.} \end{array} \right.$$
It's not hard to check that $\algd$ with $\cdot$ satisfies the hypotheses of Theorem~\ref{mainthm}, with the exception of (2). Now we construct a \adj{} as follows:
\begin{enumerate}
\item The variable set is $X=\{w,x,y,z\}$
\item $\algp_u=\algd$ for each $u\in X$.
\item Construct a \adj{} with potatoes all equal to $\bb{Z}_2\times\bb{Z}_2$ which does not have a solution.
\item For each $u,v\in X$, $R_{uv}$ is the relation from the (3), with $(\top,\top)$ added to it.
\end{enumerate}

This \adj{} has a solution given by $\varphi(u)=\top$ for each $u\in x$. The only Bulatov solution of the quotient system gives rise to the induced subinstance found in (3). This subinstance was constructed to have no solution. Therefore, Algorithm~\ref{mainalgorithm} will output NO, which is incorrect by the first sentence of this paragraph. The problem is that $\algd$ fails Lemma~\ref{functionslem}, so the existence of a solution does not guarantee the existence of solutions through all Bulatov solutions to the quotient system.

However, it can stlll be shown that $\csp(\algd)$ has a polynomial time algorithm by some unpublished work of Mckenzie, Markovic, and Maroti. They have collectively shown, in some cases, including this one, that $\algd$ is tractable if it has a quotient with a semilattice operation whose associated order is a tree, and each block is in a variety with a ``Maltsev" operation.

Freese and Mckenzie have shown in an unpublished paper that the Maltsev product of two Taylor varieties has a Taylor operation. Therefore, if the Algebraic Dichotomy Conjecture holds, the Maltsev product of two similar tractable varieties should be a tractable class. We finish off with three open problems.

\begin{probs}
Let $\cc{A}$ and $\cc{B}$ be tractable varieties. Are $\cc{A}\join\cc{B}$ or $\cc{A}\circ\cc{B}$ tractable if
\begin{enumerate}
\item $\cc{A}$ has a Maltsev (or edge) term and $\cc{B}$ has a semilattice (or $2$-semilattice) term.
\item $\cc{A}$ has a Maltsev (or edge) term and $\cc{B}$ is congruence meet-semidistributive.
\item no other assumptions are made about $\cc{A}$ and $\cc{B}$.
\end{enumerate}
\end{probs}

\newpage

\bibliography{fall2015}
\bibliographystyle{amsplain}
\end{document}